\documentclass[12pt,reqno]{amsart}
\usepackage{graphicx}
\usepackage{epsfig}
\renewcommand{\(}{\left\(}
\renewcommand{\)}{\right\)}
\renewcommand{\[}{\left\[}
\renewcommand{\]}{\right\]}

\numberwithin{equation}{section}
 \theoremstyle{plain}
\newtheorem{theorem}{Theorem}[section]
\newtheorem{lemma}[theorem]{Lemma}

\newtheorem{conjecture}[theorem]{Conjecture}

   \makeatletter
\def\proof{\@ifnextchar[{\@oproof}{\@nproof}}
\def\@oproof[#1][#2]{\trivlist\item[\hskip\labelsep\textit{#2 Proof of\
#1.}~]\ignorespaces}
\def\@nproof{\trivlist\item[\hskip\labelsep\textit{Proof.}~]\ignorespaces}

\makeatother

\begin{document}
\title[Convexity of quotients of theta functions]{Convexity of quotients of theta functions}
\author{Atul Dixit, Arindam Roy and Alexandru Zaharescu}\thanks{
The third author is supported by NSF grant number DMS - 0901621.}
\address{Department of Mathematics, University of Illinois, 1409 West Green
Street, Urbana, IL 61801, USA} \email{aadixit2@illinois.edu, roy22@illinois.edu, zaharesc@math.uiuc.edu}
\begin{abstract}
For fixed $u$ and $v$ such that $0\leq u<v<1/2$, the monotonicity of the quotients of Jacobi theta functions, namely, $\theta_{j}(u|i\pi t)/\theta_{j}(v|i\pi t)$, $j=1, 2, 3, 4$, on $0<t<\infty$ has been established in the previous works of A.Yu.~Solynin, K.~Schiefermayr, and Solynin and the first author. In the present paper, we show that the quotients $\theta_{2}(u|i\pi t)/\theta_{2}(v|i\pi t)$ and $\theta_{3}(u|i\pi t)/\theta_{3}(v|i\pi t)$ are convex on $0<t<\infty$.
\end{abstract}
\maketitle
\section{Introduction}
Let $q=e^{\pi i\tau}$ with Im $\tau>0$. The Jacobi theta functions are defined by \cite[p.~355, Section 13.19]{htf}
\begin{align*}\label{thetaf}  %
\theta_1(z|\tau) &= 2\sum_{n=0}^{\infty} (-1)^n
q^{(n+\frac{1}2)^2} \sin(2n+1)\pi z, \\  
\theta_2(z|\tau) &= 2\sum_{n=0}^{\infty} q^{(n+\frac{1}2)^2}
\cos(2n+1)\pi z,\\
\theta_3(z|\tau) &= 1+2\sum_{n=1}^{\infty} q^{n^2} \cos2n\pi z,\\
\theta_4(z|\tau) &= 1+2\sum_{n=1}^{\infty} (-1)^n q^{n^2}
\cos2n\pi z.
\end{align*}
We denote $\theta_{i}(z|\tau)$ by $\theta_{i}(z)$, $i=1, 2, 3$ and $4$, when the dependence on $z$ is to be emphasized and that on $\tau$ is to be suppressed. Moreover when $z=0$, we denote the above theta functions by $\theta_{i}$, i.e., $\theta_{i}:=\theta_{i}(0|\tau), i=1, 2, 3$ and $4$, where it is easy to see that $\theta_1=0$.

For $u,v\in \mathbb{C}$ and $\tau=i\pi t$ with Re $t>0$, define $S_j(u,v;t)$, $j=1, 2, 3$ and $4$, to be the following quotient of theta functions:
\begin{equation}\label{sdef}
S_j:=S_j(u,v;t):=\frac{\theta_j(u/2|i\pi t)}{\theta_j(v/2|i\pi t)}.
\end{equation}
Monotonicity of these quotients has attracted a lot of attention in recent years. Monotonicity of $S_{2}(u,v;t)$ on $0<t<\infty$ arose naturally in the work of A.Yu.~Solynin \cite{sol} where it is related to the steady-state distribution of heat. In particular, Solynin used it to prove a special case of a generalization of a conjecture due to A.A.~Gonchar \cite[Problem 7.45]{gp} posed by A.~Baernstein II \cite{alb}. (For complete history and progress on Gonchar's conjecture, the reader should consult \cite{bet,dixsoly}). However, the proof for $S_{2}(u,v;t)$ in \cite{sol} contained a small error. This was rectified by A.Yu.~Solynin and the first author in \cite{dixsoly}, where they also proved monotonicity of $S_1(u,v;t), S_{3}(u,v;t)$ and $S_{4}(u,v;t)$. However, it turns out that K.~Schiefermayr \cite[Theorem 1]{sch} obtained the same results as those in \cite{dixsoly} on monotonicity of $S_{3}(u,v;t)$ and $S_{4}(u,v;t)$ two years before the appearance of \cite{dixsoly}, though the proofs in \cite{dixsoly} and \cite{sch} use entirely different ideas. These results on monotonicity of $S_{j}(u,v;t), j=1, 2, 3, 4$, are stated in \cite{dixsoly} as follows.\\

\textit{For fixed $u$ and $v$ such that $0\le u<v<1$, the functions $S_1(u,v;t)$ and $S_4(u,v;t)$  are positive
and strictly increasing on $0<t<\infty$, while the functions
$S_2(u,v;t)$ and $S_3(u,v;t)$ are positive and strictly decreasing on $0<t<\infty$.}\\

At the end of the paper \cite{dixsoly}, based on numerical calculations, it was conjectured that $S_j(u,v;t)$, $j=1, 2, 3, 4$, are completely monotonic on $0<t<\infty$. A function $f$ is said to be completely monotonic on $[0,\infty)$ if $f\in C[0,\infty)$, $f\in C^{\infty}(0,\infty)$ and $(-1)^{k}f^{(k)}(t)\geq 0$ for any $k$ non-negative and $t>0$. Several functions related to gamma function, digamma function, polygamma function and modified Bessel function etc. have been shown to be completely monotonic. See \cite{chen, qgg, ism}. For a survey on properties of completely monotonic functions, see \cite{ms}. The above-mentioned conjecture can be precisely formulated (and corrected) as follows.
\begin{conjecture}\label{conj}
Let $S_j(u,v;t)$ be defined in \textup{(\ref{sdef})}. For fixed $u$ and $v$ such that $0\le u<v<1$, the functions $\frac{\partial}{\partial t}S_1(u,v;t), S_2(u,v;t), S_3(u,v;t)$ and $\frac{\partial}{\partial t}S_4(u,v;t)$ are completely monotonic on $0<t<\infty$.
\end{conjecture}
If this conjecture is indeed true, by a theorem of S.N.~Bernstein and D.~Widder \cite[p.~95, Theorem 1]{wawi} there exist non-decreasing bounded functions $\gamma_{j}$ such that $S_{j}(u,v;t)=\int_{0}^{\infty}e^{-st}d\gamma_{j}(s)$ for $j=2, 3,$ and $\frac{\partial}{\partial t}S_{j}(u,v;t)=\int_{0}^{\infty}e^{-st}d\gamma_{j}(s)$ for $j=1,4$.

In the present paper, we study convexity of $S_{2}(u,v;t)$ and $S_{3}(u,v;t)$ as functions of $t$. Figures 1 and 2 seem to indicate that these quotients are convex on $0<t<\infty$, which is consistent with the above conjecture. 
\begin{figure}
\centering
\begin{tabular}{cc}
\epsfig{file=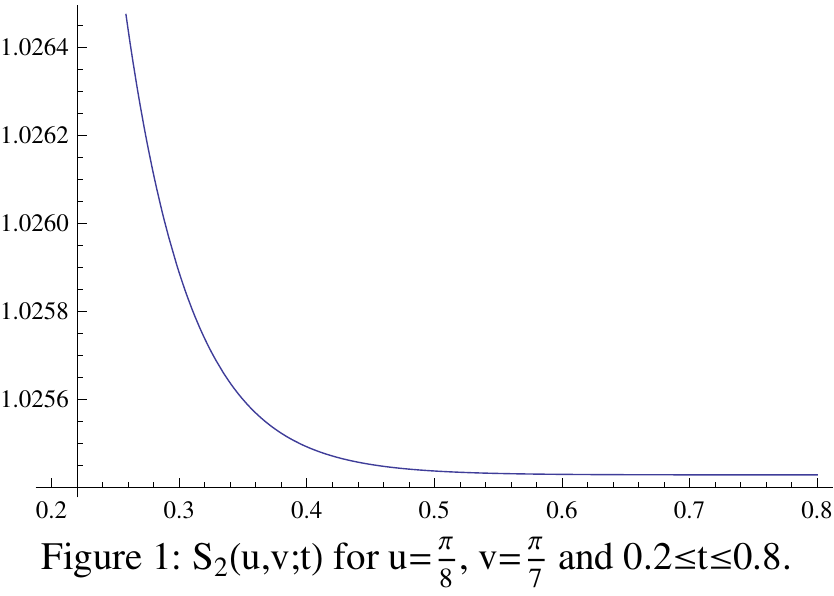, width=0.5\linewidth,clip=}&
\epsfig{file=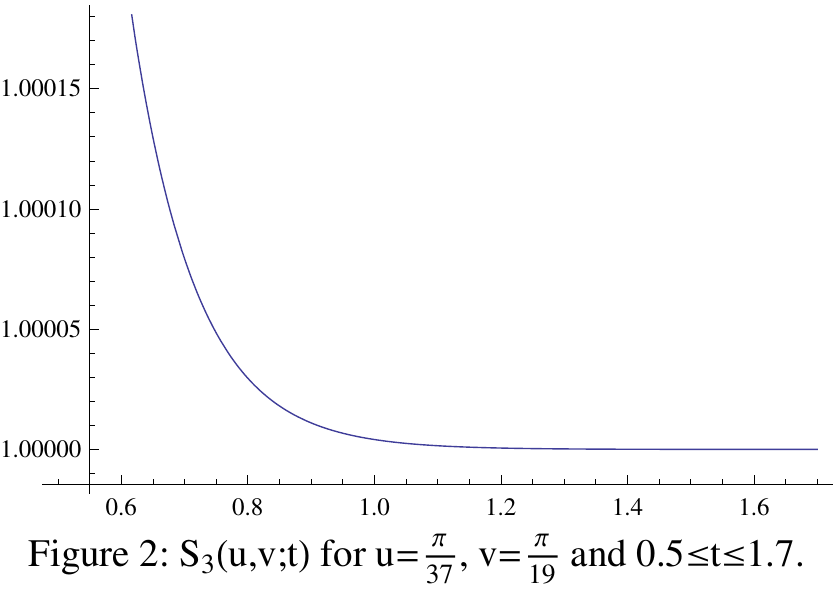, width=0.5\linewidth,clip=}
\end{tabular}
\end{figure}
Our main result given below shows that this is indeed true.  
\begin{theorem}\label{drz}
For fixed $u$ and $v$ such that $0\leq u<v<1$, the functions $S_{2}$ and $S_{3}$ are strictly convex on $0<t<\infty$.
In other words, $\frac{\partial S_{2}}{\partial t}$ and $\frac{\partial S_{3}}{\partial t}$
are negative and strictly increasing on $0<t<\infty$.
\end{theorem}
\section{Preliminary results}
In this section, we collect main ingredients all of which are subsequently required in the proofs of our results. We then prove certain lemmas also to be used in the later sections. Then in Section $3$, we prove Theorem \ref{drz} for $\frac{\partial S_{2}}{\partial t}$. 
Finally, Section $4$ is devoted to the proof of Theorem \ref{drz} for $\frac{\partial S_{3}}{\partial t}$.

We first start with some important properties of Weierstrass elliptic function. For $z\in\mathbb{C}$, let $\wp(z)$ denote the Weierstrass elliptic function with periods $1$ and $\tau$. It is known \cite[p.~376]{htf} that $\wp(z)$ maps the period parallelogram $R$ (rectangle in our case) with vertices $0$,  $\omega=1/2$, $\omega+\omega'=1/2+\tau/2$ and $\omega'=\tau/2$ conformally and one-to-one onto the lower half plane $\{\omega:\text{Im}\hspace{0.5mm}\omega<0\}$. Moreover, $\wp(z)$ is real and decreases from $\infty$ to $-\infty$ as $z$ describes the boundary of $R$ in the counterclockwise direction starting from $0$. It is known that $\wp(z)$ and $\wp'(z)$ are respectively even and odd functions of $z$.

Let $g_2$ and $g_3$ denote the invariants of $\wp(z)$. The following differential equations for $\wp$ are well-known and can be found in \cite[p.~332]{htf}:
\begin{align}\label{wpde}
{\wp'}^2(z)&=4\wp^3(z)-g_2\wp(z)-g_3,\nonumber\\
\wp''(z)&=6\wp^2(z)-\frac{g_2}{2},\nonumber\\
\wp'''(z)&=12\wp(z)\wp'(z).
\end{align}
The first equation in (\ref{wpde}) can also be represented in the form \cite[p.~331]{htf}
\begin{equation}\label{wpdea}
{\wp'}^2(z)=4\left(\wp(z)-e_1\right)\left(\wp(z)-e_2\right)\left(\wp(z)-e_3\right),
\end{equation}
where $e_1, e_2$ and $e_3$ are values of the $\wp(z)$ at $z=1/2,(\tau +1)/2$ and $\tau/2$ respectively \cite[p.~330]{htf}. As can be easily seen from (\ref{wpdea}), $\wp'(z)$ vanishes at these values of $z$. It is known that $e_3<e_2<e_1$, that $e_3<0$ and that $e_1>0$. Again, from \cite[p.~332]{htf}, we find that
\begin{align}\label{e1e2e3}
e_1&=-e_2-e_3\nonumber\\
g_2&=-4(e_1e_2+e_2e_3+e_3e_1)\nonumber\\
g_3&=4e_1e_2e_3.
\end{align}
Further, the quantities $e_1, e_2$ and $e_3$ are related to theta functions by \cite[p.~361]{htf}
\begin{align}\label{rel1}
(e_1-e_3)^{1/2}&=\pi\theta_{3}^2,\nonumber\\
(e_1-e_2)^{1/2}&=\pi\theta_{4}^2.
\end{align}
An important quantity which arises while expressing $\wp(z)$ in terms of theta functions is the following multiple of weight $2$ Eisenstein series \cite[p.~87, Equation 4.1.7]{spirit} given by
\begin{equation}\label{defc0}
c_{0}=c_{0}(q)=-\frac{\pi^2}{3}\left(1-24\sum_{n=1}^{\infty}\frac{nq^n}{1-q^n}\right).
\end{equation}
See \cite{dixsoly}. Using \cite[Equation 4.4]{dixsoly}, we have
\begin{equation}\label{ineq}
e_3<c_0<e_2<e_1.
\end{equation}
We note that $\theta_{2}(x|i\pi t)$ and $\theta_{3}(x|i\pi t)$ are related to $\theta_{1}(x|i\pi t)$ by following simple relations:
\begin{align}\label{rel}
\theta_{2}(x|i\pi t)&=\theta_{1}(1/2-x|i\pi t),\nonumber\\
\theta_{3}(x|i\pi t)&=iq^{-1/4}e^{-i\pi x}\theta_{1}(x|i\pi t).
\end{align}
Observe that from \cite[Equation (2.9)]{dixsoly}, we have on $0<x<1/2$,
\begin{equation*}\label{pp}
2\frac{\theta_{1}'(x)}{\theta_{1}(x)}+\frac{\wp'(x)}{\wp(x)-c_{0}}>0,
\end{equation*}
which when combined with (\ref{rel}) implies that on $0<x<1/2$,
\begin{equation}\label{nu1}
2\frac{\theta_{2}'(x)}{\theta_{2}(x)}+\frac{\wp'(x-1/2)}{\wp(x-1/2)-c_{0}}<0.
\end{equation}
Finally, we use the fact that each of the theta functions $\theta_{j}(x/2|i\pi t)$, $j=1, 2, 3$ and $4$, satisfies the heat equation \cite[Section 13.19]{htf}
\begin{equation}\label{he}
\frac{\partial\theta}{\partial t}=\frac{\partial^{2}\theta}{\partial x^2}.
\end{equation}
We now prove an inequality which will be instrumental in our proof of monotonicity of $S_2$ on $0<t<\infty$.
\begin{lemma}\label{ineq1}
Let $0<q<1$. Let $e_1, g_2, g_3\mbox{ and } c_0$ be defined as above. Then the following inequality holds:
\begin{equation}\label{exp1}
e_{1}^{2}(g_2-12c_{0}^{2})+e_1(6g_3+4g_2 c_0)+\left(\frac{g_2^2}{4}+g_2c_0^2+6g_3c_{0}\right)<0.
\end{equation}
\end{lemma}
\begin{proof}
Let $T(q)$ denote the left-hand side of (\ref{exp1}). We view $T(q)$ as a quadratic function in $c_0$ rather than that in $e_1$, i.e.,
\begin{equation}\label{exp2}
T(q)=(g_2-12e_1^{2})c_0^2+(6g_3+4g_2e_1)c_0+\left(\frac{g_2^2}{4}+g_2e_1^2+6g_3e_{1}\right).
\end{equation}
Employing (\ref{e1e2e3}) in (\ref{exp2}), we see that
\begin{align}\label{exp21}
T(q)&=-4(2e_{2}^{2}+5e_2e_3+2e_{3}^{2})c_{0}^{2}-8(2e_{2}^3+7e_{2}^2e_3+7e_2e_{3}^3+2e_{3}^3)c_0\nonumber\\
&\quad+(8e_{2}^4+44e_{2}^3e_3+76e_{2}^2e_{3}^2+44e_{2}e_{3}^3+8e_{3}^4)\nonumber\\
&=-4(2e_2+e_3)(e_2+2e_3)(c_{0}^{2}+2(e_2+e_3)c_{0}-(e_{2}^2+3e_2e_3+e_{3}^2)).
\end{align}
The quadratic in $c_0$ in the last expression in (\ref{exp21}) has discriminant
\begin{equation*}
4(e_2+e_3)^2+4(e_2^2+3e_2e_3+e_3^2)=4(2e_2+e_3)(e_2+2e_3)=4(e_1-e_2)(e_1-e_3),
\end{equation*}
where we utilized (\ref{e1e2e3}) in the last equality. Hence,
\begin{align}\label{exp3}
T(q)=&-4(e_1-e_2)(e_1-e_3)\left(c_0-\left(-(e_2+e_3)+\pi ^2\theta_{3}^2\theta_{4}^2\right)\right)\left(c_0-\left(-(e_2+e_3)-\pi ^2\theta_{3}^2\theta_{4}^2\right)\right)\nonumber\\
&\quad=-4(e_1-e_2)(e_1-e_3)(c_0-e_1-\pi^2\theta_{3}^2\theta_{4}^2)(c_0-e_1+\pi^2\theta_{3}^2\theta_{4}^2),
\end{align} 
where we invoked (\ref{rel1}) in the first equality and (\ref{e1e2e3}) in the second.
Using (\ref{ineq}) and (\ref{exp3}), it suffices to show that $e_1-c_0>\pi^2\theta_3^2\theta_4^2$. To that end, observe that using \cite[p.~15, Equation (1.3.32)]{spirit}, we have
\begin{equation}\label{th4}
\theta_{3}\theta_{4}=\theta_{4}^{2}(0|2\tau).
\end{equation}
Also, from \cite[Equation 4]{hir},
\begin{equation}\label{4p}
\theta_{4}^{4}=1+8\sum_{n=1}^{\infty}\frac{(-1)^nq^n}{(1+q^n)^2}.
\end{equation}
Using (\ref{th4}) and (\ref{4p}), we deduce that
\begin{equation}\label{bf1}
\pi^2\theta_3^2\theta_4^2=\pi^2+8\pi^2\sum_{n=1}^{\infty}\frac{(-1)^nq^{2n}}{(1+q^{2n})^2}.
\end{equation}
But from \cite[Equation 4.1]{dixsoly},
\begin{equation}\label{bf2}
e_1-c_0=\pi^2+8\pi^2\sum_{n=1}^{\infty}\frac{q^{2n}}{(1+q^{2n})^2}.
\end{equation}
Thus (\ref{bf1}) and (\ref{bf2}) along with the fact that $0<q<1$ imply the inequality $e_1-c_0>\pi\theta_3^2\theta_4^2$. This proves (\ref{exp1}).
\end{proof}
\begin{lemma}\label{ineqq2}
Let $0<q<1$. Let $e_2, g_2, g_3\mbox{ and } c_0$ be defined as above. Then the following inequality holds:
\begin{equation}\label{expp2}
e_{2}^{2}(g_2-12c_{0}^{2})+e_2(6g_3+4g_2 c_0)+\left(\frac{g_2^2}{4}+g_2c_0^2+6g_3c_{0}\right)>0.
\end{equation}
\end{lemma}
\begin{proof}
Let $U(q)$ denote the left-hand side of (\ref{expp2}). From (\ref{e1e2e3}) and (\ref{ineq}), 
\begin{align*}\label{expp3}
U(q)&=(g_2-12e_2^{2})c_0^2+(6g_3+4g_2e_2)c_0+\left(\frac{g_2^2}{4}+g_2e_2^2+6g_3e_{2}\right)\nonumber\\
&=-4(e_2-e_3)(2e_2+e_3)(c_0^2-2e_2c_0-(e_2^2-e_2e_3-e_3^2))\nonumber\\
&=4(e_1-e_2)(e_2-e_3)((c_0-e_2)^2+(e_1-e_2)(e_2-e_3))\nonumber\\
&>0.
\end{align*}
 \end{proof}
\section{Proof of monotonicity of $\displaystyle\frac{\partial S_{2}}{\partial t}$}
From \cite[Theorem 1]{dixsoly}, since $S_{2}(u,v;t)$ is decreasing on $0<t<\infty$, we see at once that $\frac{\partial S_{2}}{\partial t}<0$. Let $L_2:=\log S_{2}(u,v;t)$. Observe that 
\begin{equation}\label{drzz}
\frac{\partial S_2}{\partial t}=S_{2}\frac{\partial L_{2}}{\partial t}.
\end{equation}
In order to show that $\frac{\partial S_2}{\partial t}$ is increasing on $0<t<\infty$, it suffices to show that $\frac{\partial^{2}S_2}{\partial t^{2}}>0$. Now from (\ref{drzz}),
\begin{equation*}
\frac{\partial^{2}S_2}{\partial t^{2}}=\frac{\partial}{\partial t}\left(S_{2}\frac{\partial L}{\partial t}\right)=S_2\left(\frac{\partial^{2}L_2}{\partial t^{2}}+\left(\frac{\partial L_2}{\partial t}\right)^{2}\right).
\end{equation*}
We claim that $\frac{\partial^{2}L_2}{\partial t^{2}}>0$ whence we will be done. Using (\ref{he}) twice, we see that
\begin{equation*}
\frac{\partial^{2}}{\partial t^{2}}\theta_2(x/2|i\pi t)=\frac{\partial}{\partial t}\left(\frac{\partial^{2}}{\partial x^{2}}\theta_2(x/2|i\pi t)\right)=\frac{\partial^{2}}{\partial x^{2}}\left(\frac{\partial}{\partial t}\theta_2(x/2|i\pi t)\right)=\frac{\partial^{4}}{\partial x^{4}}\theta_2(x/2|i\pi t).
\end{equation*}
Hence,
\begin{align*}\label{secderlt}
\frac{\partial^{2}L_2}{\partial t^{2}}&=\frac{\partial}{\partial t}\left(\frac{\frac{\partial}{\partial t}\theta_2(u/2|i\pi t)}{{\theta_2(u/2|i\pi t)}}-\frac{\frac{\partial}{\partial t}\theta_2(v/2|i\pi t)}{{\theta_2(v/2|i\pi t)}}\right)\nonumber\\
&=\frac{\theta_{2}^{(4)}(u/2|i\pi t)}{\theta_{2}(u/2|i\pi t)}-\frac{\theta_{2}^{(4)}(v/2|i\pi t)}{\theta_{2}(v/2|i\pi t)}\nonumber\\
&\quad-\bigg(\bigg(\frac{\theta_{2}''(u/2|i\pi t)}{\theta_{2}(u/2|i\pi t)}\bigg)^{2}
-\bigg(\frac{\theta_{2}''(v/2|i\pi t)}{\theta_{2}(v/2|i\pi t)}\bigg)^{2}\bigg).
\end{align*}
Thus it suffices to show that the function $\theta_{2}^{(4)}(x|i\pi t)/\theta_{2}(x|i\pi t)-\left(\theta_{2}''(x|i\pi t)/\theta_{2}(x|i\pi t)\right)^{2}$ decreases on $0<x<1/2$. From now on, we fix $t$ where $0<t<\infty$ and henceforth suppress the dependence of $\theta_2(x/2|i\pi t)$ on $t$. From (\ref{rel}) and the relation \cite[Equation (2.6)]{dixsoly}
\begin{equation}\label{weipt1}
\left(\frac{\theta_{1}'(x)}{\theta_{1}(x)}\right)'=-\left(\wp(x)-c_{0}\right),
\end{equation}
we find that
\begin{equation}\label{weipt2}
\left(\frac{\theta_{2}'(x)}{\theta_{2}(x)}\right)'=-\left(\wp\left(x-1/2\right)-c_{0}\right),
\end{equation}
since $\wp(x)$ is an even function of $x$. Then by a repeated application of quotient rule for derivatives and (\ref{weipt2}), it is easy to see that the following are true:
\begin{align*}
\frac{\theta_{2}''(x)}{\theta_{2}(x)}&=\left(\frac{\theta_{2}'(x)}{\theta_{2}(x)}\right)^{2}-\left(\wp\left(x-1/2\right)-c_{0}\right),\nonumber\\
\frac{\theta_{2}'''(x)}{\theta_{2}(x)}&=\left(\frac{\theta_{2}'(x)}{\theta_{2}(x)}\right)^{3}-3\frac{\theta_{2}'(x)}{\theta_{2}(x)}\left(\wp\left(x-1/2\right)-c_{0}\right)-\wp'\left(x-1/2\right),\nonumber\\
\frac{\theta_{2}^{(4)}(x)}{\theta_{2}(x)}&=\left(\frac{\theta_{2}'(x)}{\theta_{2}(x)}\right)^{4}-6\left(\frac{\theta_{2}'(x)}{\theta_{2}(x)}\right)^{2}\left(\wp\left(x-1/2\right)-c_{0}\right)-4\frac{\theta_{2}'(x)}{\theta_{2}(x)}\wp'\left(x-1/2\right)\nonumber\\
&\quad+3\left(\wp\left(x-1/2\right)-c_{0}\right)^2-\wp''\left(x-1/2\right),
\end{align*}
from which it easily follows that
\begin{align*}
\frac{\theta_{2}^{(4)}(x)}{\theta_{2}(x)}-\left(\frac{\theta_{2}''(x)}{\theta_{2}(x)}\right)^{2}&=-4\left(\frac{\theta_{2}'(x)}{\theta_{2}(x)}\right)^{2}\left(\wp\left(x-1/2\right)-c_{0}\right)+2\left(\wp\left(x-1/2\right)-c_{0}\right)^2\nonumber\\
&\quad-4\frac{\theta_{2}'(x)}{\theta_{2}(x)}\wp'\left(x-1/2\right)-\wp''\left(x-1/2\right).
\end{align*}
Again using (\ref{weipt2}), we find that
\begin{align*}
\frac{d}{dx}\left(\frac{\theta_{2}^{(4)}(x)}{\theta_{2}(x)}-\left(\frac{\theta_{2}''(x)}{\theta_{2}(x)}\right)^{2}\right)&=8\frac{\theta_{2}'(x)}{\theta_{2}(x)}\left(\wp\left(x-1/2\right)-c_{0}\right)^2-4\left(\frac{\theta_{2}'(x)}{\theta_{2}(x)}\right)^2\wp'\left(x-1/2\right)\nonumber\\
&\quad+8\left(\wp\left(x-1/2\right)-c_{0}\right)\wp'\left(x-1/2\right)-4\frac{\theta_{2}'(x)}{\theta_{2}(x)}\wp''\left(x-1/2\right)\nonumber\\
&\quad-\wp'''\left(x-1/2\right).
\end{align*}
From the monotonicity of $\wp$ along the boundary of the rectangular lattice as mentioned in Section 2, in the case at hand, we have in particular that $\wp(x)$ is strictly decreasing on $0<x<1/2$. Hence $\wp(1/2-x)$ is strictly increasing on $0<x<1/2$. Since $\wp(1/2-x)=\wp(x-1/2)$, this implies that $\wp'(x-1/2)>0$ on $0<x<1/2$. Define the function $F_{2}(x)$ as
\begin{align}\label{capf2}
F_{2}(x)&:=\frac{1}{\wp'(x-1/2)}\frac{d}{dx}\left(\frac{\theta_{2}^{(4)}(x)}{\theta_{2}(x)}-\left(\frac{\theta_{2}''(x)}{\theta_{2}(x)}\right)^{2}\right)\nonumber\\
&=8\frac{\theta_{2}'(x)}{\theta_{2}(x)}\frac{\left(\wp\left(x-1/2\right)-c_{0}\right)^2}{\wp'\left(x-1/2\right)}-4\left(\frac{\theta_{2}'(x)}{\theta_{2}(x)}\right)^2+8\left(\wp\left(x-1/2\right)-c_{0}\right)\nonumber\\
&\quad-4\frac{\theta_{2}'(x)}{\theta_{2}(x)}\frac{\wp''\left(x-1/2\right)}{\wp'\left(x-1/2\right)}-\frac{\wp'''\left(x-1/2\right)}{\wp'\left(x-1/2\right)}.
\end{align}
It suffices to prove that $F_{2}(x)<0$. We prove this by showing that $F_{2}(1/2)=0$ and $F_{2}'(x)>0$, since then, the mean value theorem implies that for any $x\in(0,1/2)$, $F_{2}(x)-F_{2}(1/2)=F_{2}'(c)(x-1/2)$ for some $c\in(x,1/2)$. We begin by showing $F_{2}(1/2)=0$. We require the following series expansions in order to establish this. First, from \cite[p.~358, Section 13.19]{htf},
\begin{align}\label{fth}
\frac{\theta_{2}'(z)}{\theta_{2}(z)}&=-\pi\tan\pi z+4\pi\sum_{n=1}^{\infty}(-1)^n\frac{q^{2n}}{1-q^{2n}}\sin 2n\pi z\nonumber\\
&=\left(\frac{1}{z-1/2}-\frac{\pi^2}{3}(z-1/2)-\cdots\right)+4\pi\sum_{n=1}^{\infty}(-1)^n\frac{q^{2n}}{1-q^{2n}}\sin 2n\pi z.
\end{align}
Further, the Laurent series expansions of $\wp(z-1/2)$ and $\wp'(z-1/2)$ around $z=1/2$ are as follows \cite[p.~330, Section 13.12]{htf}.
\begin{align}\label{wps}
\wp(z-1/2)&=\frac{1}{(z-1/2)^2}+\frac{g_2(z-1/2)^2}{2^2.5}+\frac{g_3(z-1/2)^4}{2^2.7}+\frac{g_2^2(z-1/2)^6}{2^4.3.5^2}+...,\nonumber\\
\wp'(z-1/2)&=\frac{-2}{(z-1/2)^3}+\frac{g_2(z-1/2)}{10}+\frac{g_3(z-1/2)^3}{7}+\frac{g_2^2(z-1/2)^5}{2^3.5^2}+... .
\end{align}
Using (\ref{fth}), (\ref{wps}), the third differential equation in (\ref{wpde}) and simplifying, we find that $F_{2}(1/2)=0$. Differentiating both sides of (\ref{capf2}) with respect to $x$, using (\ref{wpde}), (\ref{weipt2}) and simplifying, we get
\begin{align}\label{capf3}
\frac{F_{2}^{'}(x)}{4}&=\frac{\theta_{2}'(x)}{\theta_{2}(x)}\cdot\frac{\wp^{2}(x-1/2)\left(g_{2}-12c_{0}^{2}\right)+\wp\left(x-1/2\right)\left(6g_{3}+4g_{2}c_{0}\right)+\left(6g_{3}c_{0}+g_{2}c_{0}^{2}+\frac{g_{2}^{2}}{4}\right)}{{\wp'}^{2}(x-1/2)}\nonumber\\
&\quad+\frac{\wp\left(x-1/2\right)\left(g_2/2-6c_{0}^{2}\right)+g_{3}+2c_{0}^{3}+g_{2}c_{0}/2}{\wp' (x-1/2)}.
\end{align}
Now we show that $F_{2}'(x)>0$. Let 
\begin{align}\label{a1}
A_{1}(x)&:=\wp(x-1/2)\left(g_2/2-6c_{0}^{2}\right)+g_{3}+2c_{0}^{3}+g_{2}c_{0}/2,\nonumber\\
A_{2}(x)&:=\wp^{2}(x-1/2)\left(g_{2}-12c_{0}^{2}\right)+\wp\left(x-1/2\right)\left(6g_{3}+4g_{2}c_{0}\right)+\left(6g_{3}c_{0}+g_{2}c_{0}^{2}+g_{2}^{2}/4\right).
\end{align}
By Remark 1 in \cite{dixsoly}, we have
\begin{equation}\label{rema}
e_1<\frac{-(2g_3+4c_0^{3}+g_2c_0)}{g_2-12c_0^2}.
\end{equation}
This along with the fact that $\wp(x-1/2)$ is strictly increasing on $0<x<1/2$ from $e_1$ to $\infty$ implies that $A_{1}$ has a unique zero, say $a_1$ in $(0,1/2)$. Now Lemma 2 from \cite{dixsoly} implies that $g_{2}-12c_{0}^{2}>0$. This along with the fact that $\wp\left(x-1/2\right)\rightarrow\infty$ as $x\rightarrow{\frac{1}{2}}^{-}$ implies that $A_{2}(x)\to\infty$ as $x\rightarrow{\frac{1}{2}}^{-}$. 
Using the fact that $\wp(1/2)=\wp(-1/2)=e_1$ and Lemma \ref{ineq1}, we have $A_{2}(0)<0$. Since $A_{2}$ is quadratic in $\wp(x-1/2)$ and $\wp(x-1/2)$ is strictly increasing on $0<x<1/2$, there exists a unique value $a_2$ of $x$ in $(0,1/2)$ such that $A_{2}(a_2)=0$. Let $P:=\wp(a_2-1/2)$. Note that $a_2$ is not a double root of $A_{2}$. Next, $P$ has two possibilities, say,
\begin{equation*}\label{p2p}
P=P_{1}:=\frac{-6g_{3}-4g_{2}c_{0}-\sqrt{\Delta}}{2(g_{2}-12c_{0}^{2})}\mbox{ or } P=P_{2}:=\frac{-6g_{3}-4g_{2}c_{0}+\sqrt{\Delta}}{2(g_{2}-12c_{0}^{2})},
\end{equation*}
where 
\begin{equation}\label{del}
\Delta:=(6g_{3}+4g_{2}c_{0})^{2}-4(g_{2}-12c_{0}^{2})(6g_{3}c_{0}+g_{2}c_{0}^{2}+g_{2}^{2}/4)>0,
\end{equation}
the last inequality coming from the above discussion. We now claim that $P=P_{2}$. Now 
\begin{equation}\label{rema1}
P_{2}>\frac{-6g_{3}-4g_{2}c_{0}}{2(g_{2}-12c_{0}^{2})}
\end{equation}
and
\begin{equation}\label{rema2}
\frac{-6g_{3}-4g_{2}c_{0}}{2(g_{2}-12c_{0}^{2})}+\frac{2g_3+4c_0^{3}+g_2c_0}{g_2-12c_0^2}=\frac{-g_{3}-g_{2}c_{0}/2-2c_0^{3}}{(g_{2}-12c_{0}^{2})}+\frac{6c_0^{3}-g_2c_0/2}{g_2-12c_0^2}>\frac{e_1-c_0}{2}>0,
\end{equation}
where we utilized (\ref{rema}) in the penultimate step and (\ref{ineq}) in the ultimate step. Therefore, by (\ref{rema}), (\ref{rema1}) and (\ref{rema2}),
\begin{equation}\label{rema3}
e_1<\frac{-(2g_3+4c_0^{3}+g_2c_0)}{g_2-12c_0^2}<P_{2}.
\end{equation}
This shows that $\wp(x-1/2)$ attains the value $P_{2}$ for a unique $x$ in the interval $(0,1/2)$. This combined with the facts that $P_{1}<P_{2}$ and $A_{2}$ has a unique root in $0<x<1/2$ implies that $P=P_{2}$.\\

\textbf{Remark 1.} The above discussion implies that $P_{1}<e_1<P_{2}$. As the real period of $\wp$ is $1$, this tells us that there is no real number $x$ such that $\wp(x-1/2)=P_{1}$.\\

Using $P=P_{2}$ and (\ref{rema3}), it is clear that $0<a_{1}<a_{2}<1/2$. Figure 3 shows the graphs of $10A_{1}(x)$\footnote{The graph of $A_{1}(x)$ is scaled by the factor of $10$ for better view without changing the fact $0<a_1<a_2<1/2$.} and $A_{2}(x)$.
\begin{figure}
\centering
\epsfig{file=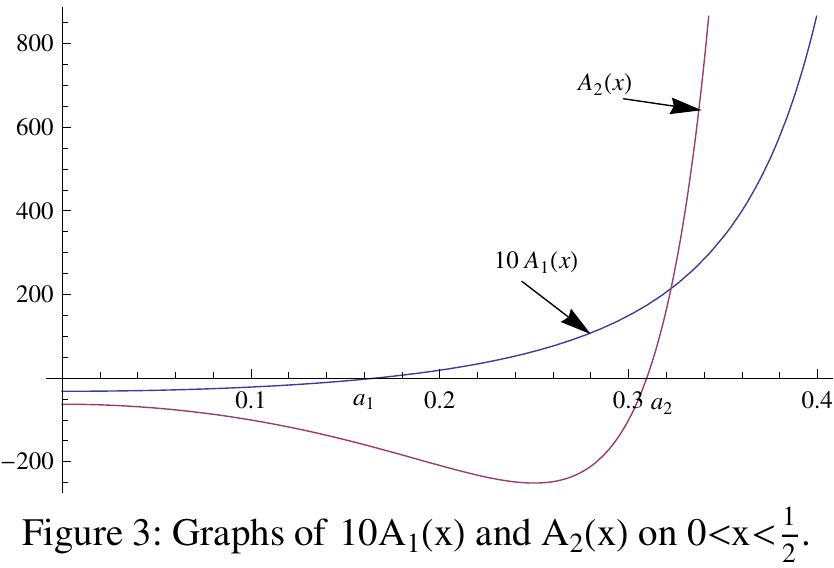, width=0.5\linewidth,clip=}
\end{figure}
Define
\begin{align}\label{G2}
G_{2}(x)&:=\frac{F_{2}'(x){\wp'}^{2}(x-1/2)}{4A_{2}(x)}\nonumber\\
&=\frac{\theta_{2}'(x)}{\theta_{2}(x)}+\frac{\wp' (x-1/2)\left(\wp\left(x-1/2\right)+\frac{2g_3+4c_{0}^{3}+g_2 c_0}{g_2-12c_{0}^{2}}\right)}{2\left(\wp^{2}(x-1/2)+\wp(x-1/2)\frac{6g_3+4g_2 c_0}{g_2-12c_{0}^{2}}+\frac{6g_{3}c_{0}+g_{2}c_{0}^{2}+g_{2}^{2}/4}{g_2-12c_{0}^{2}}\right)}.
\end{align}
Next, we differentiate extreme sides of (\ref{G2}) with respect to $x$ and use (\ref{weipt2}) so that $\theta_{2}'(x)/\theta_{2}(x)$ is eliminated from the right-hand side of (\ref{G2}) and we have everything in terms of $\wp$ and $\wp'$. This along with the second differential equation in (\ref{wpde}) gives
{\allowdisplaybreaks\begin{align}\label{capf5}
G_{2}'(x)&=-(\wp(x-1/2)-c_0)+\frac{(6\wp^{2}(x-1/2)-\frac{g_2}{2})\left(\wp(x-1/2)+\frac{2g_3+4c_{0}^{3}+g_2 c_0}{g_2-12c_{0}^{2}}\right)}{2\left(\wp^{2}(x-1/2)+\wp(x-1/2)\frac{6g_3+4g_2 c_0}{g_2-12c_{0}^{2}}+\frac{6g_{3}c_{0}+g_{2}c_{0}^{2}+g_{2}^{2}/4}{g_2-12c_{0}^{2}}\right)}\nonumber\\
&\quad+\frac{{\wp'}^{2}(x-1/2)}{2\left(\wp^{2}(x-1/2)+\wp(x-1/2)\frac{6g_3+4g_2 c_0}{g_2-12c_{0}^{2}}+\frac{6g_{3}c_{0}+g_{2}c_{0}^{2}+g_{2}^{2}/4}{g_2-12c_{0}^{2}}\right)}\nonumber\\
&\quad-\frac{{\wp'}^{2}(x-1/2)\left(\wp(x-1/2)+\frac{2g_3+4c_{0}^{3}+g_2 c_0}{g_2-12c_{0}^{2}}\right)\left(2\wp(x-1/2)+\frac{6g_3+4g_2 c_0}{g_2-12c_{0}^{2}}\right)}{2\left(\wp^{2}(x-1/2)+\wp(x-1/2)\frac{6g_3+4g_2 c_0}{g_2-12c_{0}^{2}}+\frac{6g_{3}c_{0}+g_{2}c_{0}^{2}+g_{2}^{2}/4}{g_2-12c_{0}^{2}}\right)^2}.
\end{align}}
Simplifying the first three terms of (\ref{capf5}), we obtain
\begin{align}\label{capf6}
G_{2}'(x)&=
\frac{{\wp'}^{2}(x-1/2)}{\left(\wp^{2}(x-1/2)+\wp(x-1/2)\frac{6g_3+4g_2 c_0}{g_2-12c_{0}^{2}}+\frac{6g_{3}c_{0}+g_{2}c_{0}^{2}+g_{2}^{2}/4}{g_2-12c_{0}^{2}}\right)}\nonumber\\
&\quad-\frac{{\wp'}^{2}(x-1/2)\left(\wp(x-1/2)+\frac{2g_3+4c_{0}^{3}+g_2 c_0}{g_2-12c_{0}^{2}}\right)\left(2\wp(x-1/2)+\frac{6g_3+4g_2 c_0}{g_2-12c_{0}^{2}}\right)}{2\left(\wp^{2}(x-1/2)+\wp(x-1/2)\frac{6g_3+4g_2 c_0}{g_2-12c_{0}^{2}}+\frac{6g_{3}c_{0}+g_{2}c_{0}^{2}+g_{2}^{2}/4}{g_2-12c_{0}^{2}}\right)^2}.
\end{align}

Consider three cases: $0<x<a_1$, $a_1\leq x\leq a_2$ and $a_2<x<1/2$.\\

\textbf{Case 1:} $0<x<a_1$.

Then, $A_{1}(x)<0$ and $A_{2}(x)<0$. We show that $G_{2}(x)<0$. Note that from (\ref{wpdea}), (\ref{fth}), (\ref{rema}) and Lemma \ref{ineq1}, it readily follows that $G_{2}(0)=0$. 
Since $A_{1}(x)<0$, $A_{2}(x)<0$ and $g_2-12c_0^{2}>0$, we have
\begin{align}
&\wp\left(x-1/2\right)+\frac{2g_3+4c_{0}^{3}+g_2 c_0}{g_2-12c_{0}^{2}}<0,\label{aftA1A21}\\
&\wp^{2}(x-1/2)+\wp(x-1/2)\frac{6g_3+4g_2 c_0}{g_2-12c_{0}^{2}}+\frac{6g_{3}c_{0}+g_{2}c_{0}^{2}+g_{2}^{2}/4}{g_2-12c_{0}^{2}}<0.\label{aftA1A22}
\end{align}
From (\ref{aftA1A21}) and (\ref{rema2}), we see that
\begin{align}\label{jlt}
2\wp(x-1/2)+\frac{6g_3+4g_2 c_0}{g_2-12c_{0}^{2}}<0.
\end{align}
Therefore, (\ref{aftA1A21}), (\ref{aftA1A22}) and (\ref{jlt}) imply that $G_{2}'(x)<0$. By the mean value theorem, for any $x\in(0,a_1)$, $G_{2}(x)=xG_{2}'(d)$ for some $d\in(0,x)$. Hence $G_{2}(x)<0$. Thus $F_{2}'(x)>0$ in $0<x<a_1$.\\

\textbf{Case 2:} $a_1\leq x\leq a_2$.

Note that $A_{1}(a_{1})=0$, $A_{2}(a_{1})<0$, $A_{1}(a_{2})>0$ and $A_{2}(a_{2})=0$. Also, $A_{1}(x)>0$ and $A_{2}(x)<0$ when $a_{1}<x<a_{2}$. 

Since $\wp(x-1/2)$ is strictly increasing on $0<x<1/2$, we have $\wp'(x-1/2)>0$ and $\wp(x-1/2)-c_0>e_1-c_0>0$, where we invoked (\ref{ineq}) in the last step. This along with (\ref{nu1}) shows that $\theta_{2}'(x)/\theta_{2}(x)<0$ on $0<x<1/2$.  Using all of the above facts and (\ref{capf3}), we observe that $F_{2}'(x)>0$ on $a_{1}\leq x\leq a_{2}$.\\

\textbf{Case 3:} $a_2<x<1/2$.
Since $A_{1}(x)>0$, $A_{2}(x)>0$ and $g_2-12c_0^{2}>0$, we have
\begin{align}
&\wp\left(x-1/2\right)+\frac{2g_3+4c_{0}^{3}+g_2 c_0}{g_2-12c_{0}^{2}}>0,\label{aftA1A23}\\
&\wp^{2}(x-1/2)+\wp(x-1/2)\frac{6g_3+4g_2 c_0}{g_2-12c_{0}^{2}}+\frac{6g_{3}c_{0}+g_{2}c_{0}^{2}+g_{2}^{2}/4}{g_2-12c_{0}^{2}}>0.\label{aftA1A24}
\end{align}
From (\ref{G2}), as $x\to{\frac{1}{2}}^{-}$,
\begin{align*}
G_{2}(x)=\frac{\theta_{2}'(x)}{\theta_{2}(x)}+\frac{\wp' (x-1/2)}{2\wp(x-1/2)}\left(1+O\left(\frac{1}{\wp(x-1/2)}\right)\right).
\end{align*}
Using (\ref{fth}) and (\ref{wps}), it is easy to check that $G_{2}(1/2)=0$. Next we show that $G_{2}'(x)<0$. From (\ref{capf6}),
\begin{align*}\label{g2pn}
G_{2}'(x)=\frac{{\wp'}^{2}(x-1/2)(1-Q(x))}{\left(\wp^{2}(x-1/2)+\wp(x-1/2)\frac{6g_3+4g_2 c_0}{g_2-12c_{0}^{2}}+\frac{6g_{3}c_{0}+g_{2}c_{0}^{2}+g_{2}^{2}/4}{g_2-12c_{0}^{2}}\right)},
\end{align*}
where
\begin{equation}\label{qx}
Q(x):=\frac{\left(\wp(x-1/2)+\frac{2g_3+4c_{0}^{3}+g_2 c_0}{g_2-12c_{0}^{2}}\right)\left(2\wp(x-1/2)+\frac{6g_3+4g_2 c_0}{g_2-12c_{0}^{2}}\right)}{2\left(\wp^{2}(x-1/2)+\wp(x-1/2)\frac{6g_3+4g_2 c_0}{g_2-12c_{0}^{2}}+\frac{6g_{3}c_{0}+g_{2}c_{0}^{2}+g_{2}^{2}/4}{g_2-12c_{0}^{2}}\right)}.
\end{equation}
We claim that $Q(x)>1$. Note that the denominator of $Q(x)$ can be simplified as follows:
\begin{align}\label{r1}
\quad 2\bigg(\wp^{2}(x&-1/2)+\wp(x-1/2)\frac{6g_3+4g_2 c_0}{g_2-12c_{0}^{2}}+\frac{6g_{3}c_{0}+g_{2}c_{0}^{2}+g_{2}^{2}/4}{g_2-12c_{0}^{2}}\bigg)\nonumber\\
&=\left(2\wp(x-1/2)+\frac{6g_3+4g_2 c_0}{g_2-12c_{0}^{2}}\right)\left(\wp(x-1/2)+\frac{6g_3+4g_2 c_0}{2(g_2-12c_{0}^{2})}\right)\nonumber\\
&\quad+\left(2\frac{6g_{3}c_{0}+g_{2}c_{0}^{2}+g_{2}^{2}/4}{g_2-12c_{0}^{2}}-\frac{(6g_3+4g_2 c_0)^2}{2(g_2-12c_{0}^{2})^2}\right).
\end{align}
Now
\begin{align}\label{r2}
2\wp(x-1/2)+\frac{6g_3+4g_2 c_0}{g_2-12c_{0}^{2}}&>2\wp(a_{2}-1/2)+\frac{6g_3+4g_2 c_0}{g_2-12c_{0}^{2}}\nonumber\\
&=2P+\frac{6g_3+4g_2 c_0}{g_2-12c_{0}^{2}}\nonumber\\
&=\frac{\sqrt{\Delta}}{(g_{2}-12c_{0}^{2})}\nonumber\\
&>0.
\end{align}
From (\ref{rema2}), we have
\begin{equation}\label{r3}
\wp(x-1/2)+\frac{2g_3+4c_{0}^{3}+g_2 c_0}{g_2-12c_{0}^{2}}>\wp(x-1/2)+\frac{6g_3+4g_2 c_0}{2(g_2-12c_{0}^{2})}.
\end{equation}
By (\ref{del}), the last term on the right-hand side of (\ref{r1}) is negative. Hence, (\ref{r1}), (\ref{r2}), (\ref{r3}) and (\ref{aftA1A24}) imply that $Q(x)>1$.
Therefore $G_{2}'(x)<0$. By the mean value theorem, for any $x\in(a_2,1/2)$, $G_{2}(x)-G_{2}(1/2)=G_{2}'(b)(x-1/2)$ for some $b\in(x,1/2)$. Hence $G_{2}(x)>0$. Since $A_{2}(x)>0$, this implies that $F_{2}^{'}(x)>0$. 

From the above three cases, we conclude that $F_{2}^{'}(x)>0$ in $0<x<1/2$. Since $F_2(1/2)=0$, by another application of the mean value theorem, we conclude that $F_2(x)<0$ in $0<x<1/2$. This completes the proof.
\begin{figure}
\centering
\epsfig{file=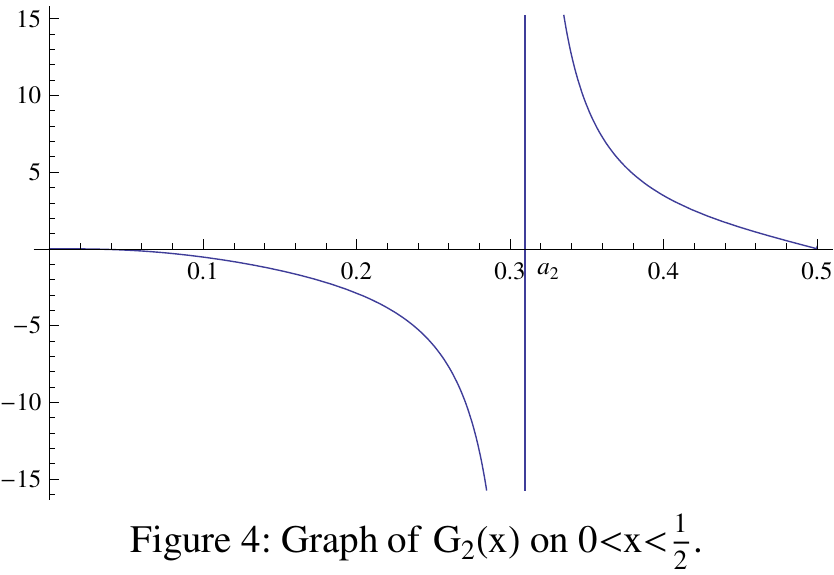, width=0.5\linewidth,clip=}
\end{figure}
Figure 4 shows the graph of $G_{2}(x)$ on $0<x<1/2$.
\section{Proof of monotonicity of $\displaystyle\frac{\partial S_{3}}{\partial t}$}
The method for proving monotonicity of  $\frac{\partial S_{3}}{\partial t}$ is similar to that of $\frac{\partial S_{2}}{\partial t}$ and so we will be brief. From \cite[Theorem 1]{dixsoly}, since $S_{3}(u,v;t)$ is decreasing on $0<t<\infty$, we see at once that $\frac{\partial S_{3}}{\partial t}<0$. Let $L_3:=\log S_{3}(u,v;t)$. Observe that 
\begin{equation*}
\frac{\partial S_3}{\partial t}=S_{3}\frac{\partial L_{3}}{\partial t}.
\end{equation*}
It suffices to show that $\frac{\partial^{2}S_3}{\partial t^{2}}>0$. Now,
\begin{equation*}
\frac{\partial^{2}S_3}{\partial t^{2}}=\frac{\partial}{\partial t}\left(S_{3}\frac{\partial L}{\partial t}\right)=S_3\left(\frac{\partial^{2}L_3}{\partial t^{2}}+\left(\frac{\partial L_3}{\partial t}\right)^{2}\right).
\end{equation*}
We show that $\frac{\partial^{2}L_3}{\partial t^{2}}>0$. Observe that using (\ref{he}) twice, we have $\frac{\partial^{2}}{\partial t^{2}}\theta_3(x/2|i\pi t)=\frac{\partial^{4}}{\partial x^{4}}\theta_3(x/2|i\pi t)$. It suffices to show that the function $\theta_{3}^{(4)}(x|i\pi t)/\theta_{3}(x|i\pi t)-\left(\theta_{3}''(x|i\pi t)/\theta_{3}(x|i\pi t)\right)^{2}$ decreases on $0<x<1/2$. Fix $t$ where $0<t<\infty$. Using (\ref{rel}) and (\ref{weipt1}), we find that
\begin{equation}\label{weipt3}
\left(\frac{\theta_{3}'(x)}{\theta_{3}(x)}\right)'=-\left(\wp\left(x+\frac{\tau-1}{2}\right)-c_{0}\right).
\end{equation}
Observe that
\begin{align*}
\frac{\theta_{3}^{(4)}(x)}{\theta_{3}(x)}-\left(\frac{\theta_{3}''(x)}{\theta_{3}(x)}\right)^{2}&=-4\left(\frac{\theta_{3}'(x)}{\theta_{3}(x)}\right)^{2}\left(\wp\left(x+\frac{\tau-1}{2}\right)-c_{0}\right)+2\left(\wp\left(x+\frac{\tau-1}{2}\right)-c_{0}\right)^2\nonumber\\
&\quad-4\frac{\theta_{3}'(x)}{\theta_{3}(x)}\wp'\left(x+\frac{\tau-1}{2}\right)-\wp''\left(x+\frac{\tau-1}{2}\right).
\end{align*}
Using (\ref{weipt3}), we find that
\begin{align*}
\frac{d}{dx}\left(\frac{\theta_{3}^{(4)}(x)}{\theta_{3}(x)}-\left(\frac{\theta_{3}''(x)}{\theta_{3}(x)}\right)^{2}\right)&=8\frac{\theta_{3}'(x)}{\theta_{3}(x)}\left(\wp\left(x+\frac{\tau-1}{2}\right)-c_{0}\right)^2-4\left(\frac{\theta_{3}'(x)}{\theta_{3}(x)}\right)^2\wp'\left(x+\frac{\tau-1}{2}\right)\nonumber\\
&\quad+8\left(\wp\left(x+\frac{\tau-1}{2}\right)-c_{0}\right)\wp'\left(x+\frac{\tau-1}{2}\right)\nonumber\\
&\quad-4\frac{\theta_{3}'(x)}{\theta_{3}(x)}\wp''\left(x+\frac{\tau-1}{2}\right)-\wp'''\left(x+\frac{\tau-1}{2}\right).
\end{align*}
Since $\wp\left(x+\frac{\tau-1}{2}\right)$ decreases on $0<x<1/2$, we have $\wp'\left(x+\frac{\tau-1}{2}\right)<0$. Define a function $F_{3}(x)$ as
\begin{align}\label{3capf}
F_{3}(x)&:=\frac{1}{\wp'(x+\frac{\tau-1}{2})}\frac{d}{dx}\left(\frac{\theta_{3}^{(4)}(x)}{\theta_{3}(x)}-\left(\frac{\theta_{3}''(x)}{\theta_{3}(x)}\right)^{2}\right)\nonumber\\
&=8\frac{\theta_{3}'(x)}{\theta_{3}(x)}\frac{\left(\wp\left(x+\frac{\tau-1}{2}\right)-c_{0}\right)^2}{\wp'\left(x+\frac{\tau-1}{2}\right)}-4\left(\frac{\theta_{3}'(x)}{\theta_{3}(x)}\right)^2+8\left(\wp\left(x+\frac{\tau-1}{2}\right)-c_{0}\right)\nonumber\\
&\quad-4\frac{\theta_{3}'(x)}{\theta_{3}(x)}\frac{\wp''\left(x+\frac{\tau-1}{2}\right)}{\wp'\left(x+\frac{\tau-1}{2}\right)}-\frac{\wp'''\left(x+\frac{\tau-1}{2}\right)}{\wp'\left(x+\frac{\tau-1}{2}\right)}.
\end{align}
It suffices to prove that $F_{3}(x)>0$. We prove this by showing that $F_{3}'(x)<0$ and $F_{3}(1/2)>0$, because then by the mean value theorem, for any $x\in(0,1/2)$, we have $F_{3}(x)-F_{3}(1/2)=F_{3}'(e)(x-1/2)$ for some $e\in(x,1/2)$ whence $F_{3}(x)>0$. We first show that $F_{3}(1/2)>0$. Using the thirs differential equation in (\ref{wpde}), we have
\begin{align}\label{l3capf}
F_{3}(1/2)&=8(e_3-c_0)^2\lim_{x\to {\frac{1}{2}}^{-}}\frac{\theta_{3}'(x)/\theta_{3}(x)}{\wp'\left(x+\frac{\tau-1}{2}\right)}-4\lim_{x\to {\frac{1}{2}}^{-}}\left(\frac{\theta_{3}'(x)}{\theta_{3}(x)}\right)^2+8(e_3-c_{0})\nonumber\\
&\quad-4\wp''(\tau/2)\lim_{x\to {\frac{1}{2}}^{-}}\frac{\theta_{3}'(x)/\theta_{3}(x)}{\wp'\left(x+\frac{\tau-1}{2}\right)}-12e_3.
\end{align}
Now \cite[p.~358, Section 13.19]{htf}
\begin{align}\label{fth3}
\frac{\theta_{3}'(z)}{\theta_{3}(z)}=4\pi\sum_{n=1}^{\infty}(-1)^n\frac{q^{n}}{1-q^{2n}}\sin 2n\pi z
\end{align} 
implies that $\theta_{3}'(x)/\theta_{3}(x)$ vanishes at $x=1/2$. Note that $\wp'\left(x+\frac{\tau-1}{2}\right)=0$ at $x=1/2$ too. Hence, using L'Hopital's rule in (\ref{l3capf}), then (\ref{weipt3}), the second differential equation in (\ref{wpde}) and simplifying, we see that
\begin{equation*}
F_{3}(1/2)=\frac{16(e_3-c_0)^3}{g_2-12e_3^2}-12c_0.
\end{equation*}
Now using (\ref{e1e2e3}) and (\ref{ineq}), note that 
\begin{align*}
g_2-12e_3^2&=-4(e_1e_2+e_2e_3+e_3e_1)-12e_3^2\nonumber\\
&=4(e_3-e_1)(e_2-e_3)\nonumber\\
&<0.
\end{align*}
Thus, we need to show that $16(e_3-c_0)^3-12c_0(g_2-12e_3^2)<0$ or equivalently, $(e_3-c_0)^3<3c_0(e_3-e_1)(e_2-e_3)$. Consider two cases.

\textbf{Case 1}: $c_0\leq 0$.

By (\ref{ineq}), the left-hand side is less than zero but the right-hand side is greater than or equal to zero. This proves the required inequality.

\textbf{Case 2}: $c_0>0$.

Using (\ref{e1e2e3}),
\begin{align*}
&3c_0(e_3-e_1)(e_2-e_3)-(e_3-c_0)^3\nonumber\\
&=(e_1+e_2+c_0)^3-3c_0(2e_1+e_2)(e_1+2e_2)\nonumber\\
&=\frac{1}{27}\left(((2e_1+e_2)+(e_1+2e_2)+3c_0)^3-27\cdot 3c_0(2e_1+e_2)(e_1+2e_2)\right).
\end{align*}
The last expression is clearly positive by the Arithmetic mean-Geometric mean inequality and since $2e_1+e_2$, $e_1+2e_2$ are positive by (\ref{ineq}) and since $3c_0$ is positive. From the above two cases, we conclude that $F_{3}(1/2)>0$. Our next task is to show that $F_{3}'(x)<0$. From (\ref{3capf}), we have
\begin{align*}
\frac{F_{3}^{'}(x)}{4}=\frac{\theta_{3}'(x)}{\theta_{3}(x)}\frac{A_{2}(x+\frac{\tau}{2})}{{\wp'}^{2}\left(x+\frac{\tau-1}{2}\right)}+\frac{A_{1}(x+\frac{\tau}{2})}{\wp'\left(x+\frac{\tau-1}{2}\right)},
\end{align*}
where $A_{1}(x)$ and $A_{2}(x)$ are defined in (\ref{a1}). Now
\begin{equation*}
A_{2}'\left(x+\frac{\tau}{2}\right)=\wp'\left(x+\frac{\tau-1}{2}\right)\left(2\left(g_{2}-12c_{0}^{2}\right)\wp\left(x+\frac{\tau-1}{2}\right)+\left(6g_{3}+4g_{2}c_{0}\right)\right).
\end{equation*}
From (\ref{ineq}), (\ref{rema}) and the facts that $e_3<\wp\left(x+\frac{\tau-1}{2}\right)<e_2$ and $\wp'\left(x+\frac{\tau-1}{2}\right)<0$ on $0<x<1/2$, we find that $A_{2}'(x+\frac{\tau}{2})>0$. Also by Lemma \ref{ineqq2}, $A_{2}(\frac{\tau}{2})>0$. By the mean value theorem, for any $x\in(0,1/2)$, we have $A_{2}(x+\frac{\tau}{2})=A_{2}(\frac{\tau}{2})+xA_{2}'(k+\frac{\tau}{2})>0$ for some $k\in(0,x)$. 
\begin{figure}
\centering
\epsfig{file=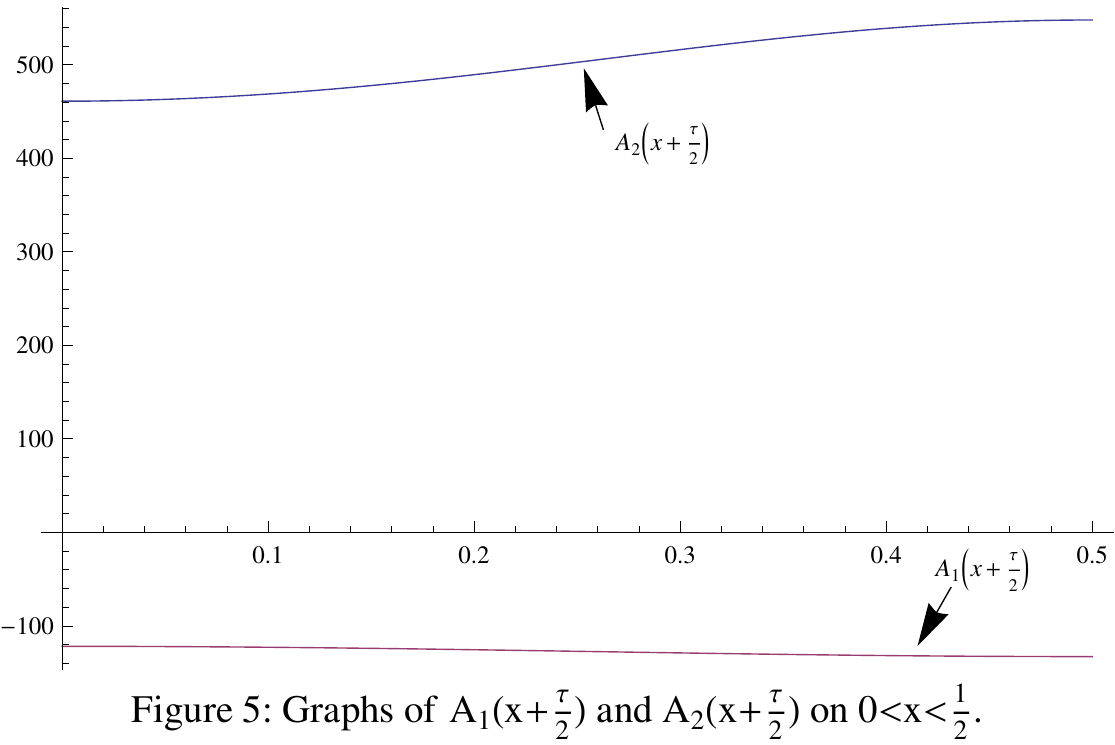, width=0.5\linewidth,clip=}
\end{figure}
Figure 5 shows the graphs of $A_{1}(\frac{\tau}{2})$ and $A_{2}(\frac{\tau}{2})$ on $0<x<1/2$. Now define $G_{3}$ by
{\allowdisplaybreaks\begin{align}\label{G3}
G_{3}(x)&:=\frac{F_{3}'(x){\wp'}^{2}(x+\frac{\tau-1}{2})}{4A_{2}(x+\frac{\tau}{2})}\nonumber\\
&=\frac{\theta_{3}'(x)}{\theta_{3}(x)}+\frac{\wp' (x+\frac{\tau-1}{2})\left(\wp\left(x+\frac{\tau-1}{2}\right)+\frac{2g_3+4c_{0}^{3}+g_2 c_0}{g_2-12c_{0}^{2}}\right)}{2\left(\wp^{2}(x+\frac{\tau-1}{2})+\wp(x+\frac{\tau-1}{2})\frac{6g_3+4g_2 c_0}{g_2-12c_{0}^{2}}+\frac{6g_{3}c_{0}+g_{2}c_{0}^{2}+g_{2}^{2}/4}{g_2-12c_{0}^{2}}\right)}.
\end{align}}
From the above discussion, it suffices to show that $G_{3}(x)<0$. Now, from (\ref{fth3}) and the fact that $\wp'\left(\frac{\tau-1}{2}\right)=0=\wp'\left(\frac{\tau}{2}\right)$, it is easy to see that $G_{3}(0)=0=G_{3}(1/2)$. This implies that $G_{3}'(x)$ has at least one zero in $0<x<1/2$. Differentiating both sides of (\ref{G3}) with respect to $x$ and simplifying, we observe that
\begin{align*}
G_{3}'(x)=\frac{{\wp'}^{2}(x+\frac{\tau-1}{2})(1-Q(x+\frac{\tau}{2}))}{\left(\wp^{2}(x+\frac{\tau-1}{2})+\wp(x+\frac{\tau-1}{2})\frac{6g_3+4g_2 c_0}{g_2-12c_{0}^{2}}+\frac{6g_{3}c_{0}+g_{2}c_{0}^{2}+g_{2}^{2}/4}{g_2-12c_{0}^{2}}\right)},
\end{align*}
where $Q(x)$ is defined in (\ref{qx}). Now
\begin{align}\label{qcal}
1-Q\left(x+\frac{\tau}{2}\right)&=1-\frac{\left(\wp(x+\frac{\tau-1}{2})+\frac{2g_3+4c_{0}^{3}+g_2 c_0}{g_2-12c_{0}^{2}}\right)\left(2\wp(x+\frac{\tau-1}{2})+\frac{6g_3+4g_2 c_0}{g_2-12c_{0}^{2}}\right)}{2\left(\wp^{2}(x+\frac{\tau-1}{2})+\wp(x+\frac{\tau-1}{2})\frac{6g_3+4g_2 c_0}{g_2-12c_{0}^{2}}+\frac{6g_{3}c_{0}+g_{2}c_{0}^{2}+g_{2}^{2}/4}{g_2-12c_{0}^{2}}\right)}\nonumber\\
&=\frac{2\wp(x+\frac{\tau-1}{2})\frac{g_3+g_2c_0-4c_0^3}{g_2-12c_0^2}+C}{2\left(\wp^{2}(x+\frac{\tau-1}{2})+\wp(x+\frac{\tau-1}{2})\frac{6g_3+4g_2 c_0}{g_2-12c_{0}^{2}}+\frac{6g_{3}c_{0}+g_{2}c_{0}^{2}+g_{2}^{2}/4}{g_2-12c_{0}^{2}}\right)},
\end{align}
where 
\begin{equation*}
C:=\frac{2(6g_{3}c_{0}+g_{2}c_{0}^{2}+g_{2}^{2}/4)}{g_2-12c_{0}^{2}}-\frac{(6g_3+4g_2 c_0)(2g_3+4c_{0}^{3}+g_2 c_0)}{(g_2-12c_{0}^{2})^2}.
\end{equation*}
The numerator in the last expression of (\ref{qcal}) has atmost one zero since it is linear in $\wp(x+\frac{\tau-1}{2})$ and $\wp(x+\frac{\tau-1}{2})$ is monotone. Hence, $G_{3}'(x)$ has exactly one zero, say $x_{0}$, in $0<x<1/2$. Thus we will be done if we can show that $G_{3}(x)<0$ at some point in the interval $0<x<1/2$. In fact, we show that $G_{3}(x)<0$ on $(0,x_0)$.

For any $x$ in $(0,x_0)$, we have $\wp\left(x+\frac{\tau-1}{2}\right)>\wp\left(x_0+\frac{\tau-1}{2}\right)$. Also, 
\begin{align*}
\frac{g_3+g_2c_0-4c_0^3}{g_2-12c_0^2}=\frac{g_3+g_2c_0/2+2c_0^3}{g_2-12c_0^2}+\frac{c_0(g_2/2-6c_0^2)}{g_2-12c_0^2}<\frac{-(e_1-c_0)}{2}<0,
\end{align*}
where last two inequalities follows from (\ref{rema}) and (\ref{ineq}).
Therefore
\begin{equation*} {2\wp\left(x+\frac{\tau-1}{2}\right)\frac{g_3+g_2c_0-4c_0^3}{g_2-12c_0^2}+C}<{2\wp\left(x_0+\frac{\tau-1}{2}\right)\frac{g_3+g_2c_0-4c_0^3}{g_2-12c_0^2}+C}=0,
\end{equation*}
where the last equality comes from the fact that $G_{3}'(x_0)=0$. Hence, $G_{3}'(x)<0$ for $0<x<x_0$. Then it is clear by the mean value theorem that for any $x\in(0,x_0)$, $G_3(x)=xG_{3}'(x_1)<0$ for some $x_1\in(0,x)$. So finally $G_3(x)<0$ for $0<x<1/2$. This completes the proof. Figure 6 shows the graph of $G_{3}(x)$ on $0<x<1/2$.
\begin{figure}
\centering
\epsfig{file=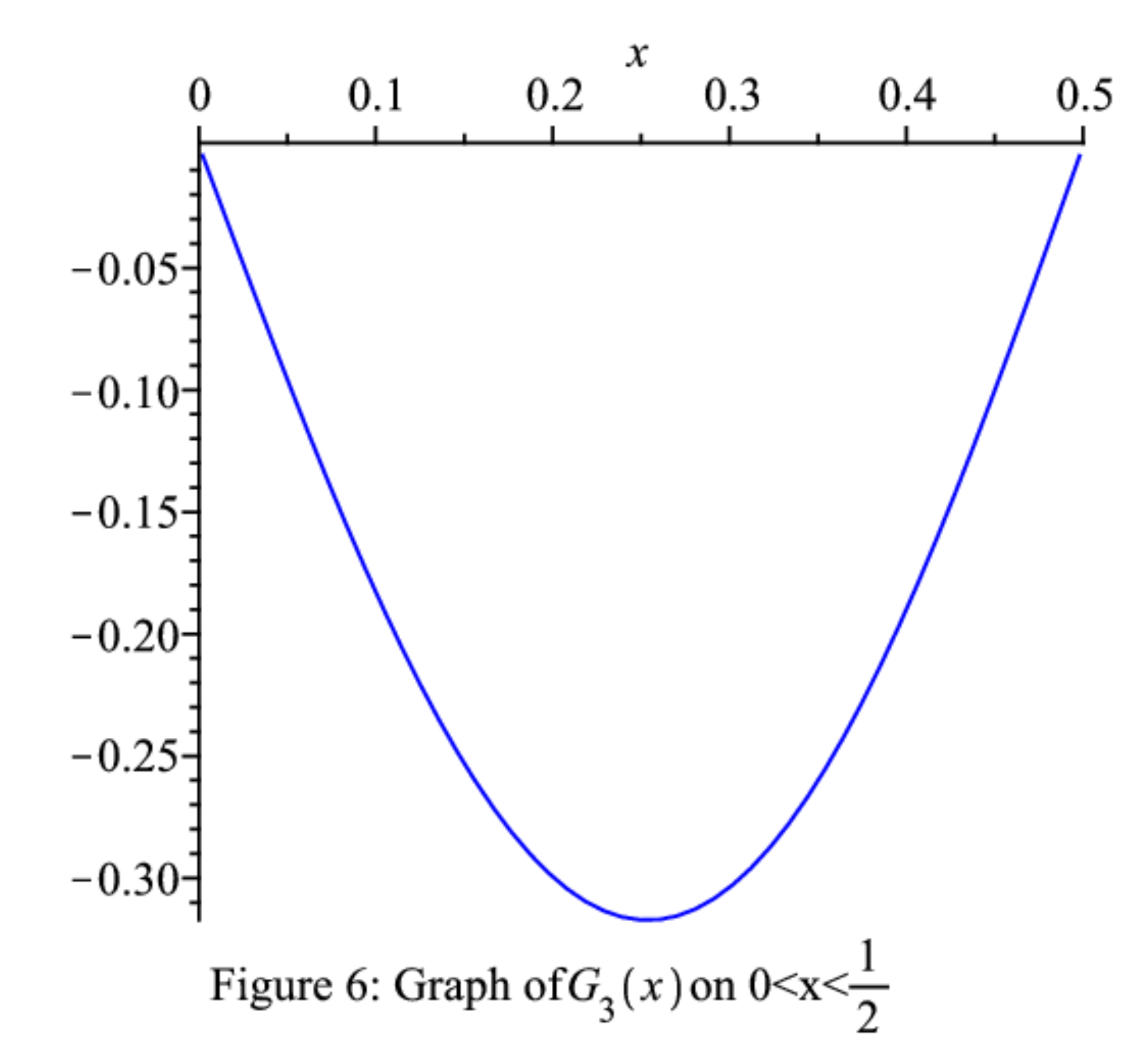, width=0.5\linewidth,clip=}
\end{figure}

\end{document}